\documentclass[11pt]{article}
\usepackage[english]{babel}
\usepackage[utf8x]{inputenc}
\usepackage[T1]{fontenc}

\usepackage[letterpaper,top=1in,bottom=1in,left=1in,right=1in,marginparwidth=1.75cm]{geometry}

\usepackage{amsmath, amssymb, amsthm, verbatim,enumerate,bbm}
\usepackage{indentfirst, bbm}
\usepackage{appendix}
\usepackage{enumerate}
\usepackage{verbatim}
\usepackage{mathrsfs}
\usepackage[colorinlistoftodos]{todonotes}
\usepackage[colorlinks=true, allcolors=blue]{hyperref}
\usepackage[nameinlink,capitalise,noabbrev]{cleveref}
\crefname{appsec}{Appendix}{Appendices}
\creflabelformat{enumi}{#2textup{#1}#3}

\date{\today}
\allowdisplaybreaks

\newtheorem{theorem}{Theorem}[section]
\newtheorem*{namedtheorem}{\theoremname}
\newcommand{\theoremname}{testing}

\newtheorem{lemma}[theorem]{Lemma}

\newtheorem{proposition}[theorem]{Proposition}

\newtheorem{corollary}[theorem]{Corollary}

\newtheorem{conjecture}[theorem]{Conjecture}
\newtheorem*{question*}{Question}

\theoremstyle{definition}
\newtheorem{definition}[theorem]{Definition}

\newtheorem{remark}[theorem]{Remark}

\theoremstyle{plain}

\usepackage{algpseudocode,algorithm,algorithmicx}

\newcommand{\Bin}{\ensuremath{\textrm{Bin}}}

\newcommand{\Ex}{\mathbb{E}}

\global\long\def\E{\mathbb{E}}

\DeclareMathOperator{\la}{la}

\title{Towards the linear arboricity conjecture}
\author{Asaf Ferber \thanks{Massachusetts Institute of Technology. Department of Mathematics. Email: {\tt ferbera@mit.edu}. Research is partially supported by NSF 6935855.} \and 
Jacob Fox \thanks{Stanford University. Department of Mathematics. Email: {\tt jacobfox@stanford.edu}. Research
supported by a Packard Fellowship and by NSF Career Award DMS-1352121.} \and
Vishesh Jain\thanks{Massachusetts Institute of Technology. Department of Mathematics. Email: {\tt visheshj@mit.edu}. Research is partially supported by NSF CCF 1665252,  NSF DMS-1737944 and ONR N00014-17-1-2598.} }
\date{}

\begin{document}
\maketitle
\begin{abstract}
The linear arboricity
 of a graph $G$, denoted by $\text{la}(G)$, is the minimum number of edge-disjoint linear forests (i.e. forests in which every connected component is a path) in $G$ whose union covers all the edges of $G$. A famous conjecture due to Akiyama, Exoo, and Harary from 1980 asserts that $\text{la}(G)\leq \lceil (\Delta(G)+1)/2 \rceil$, where $\Delta(G)$ denotes the maximum degree of $G$. This conjectured upper bound would be best possible, as is easily seen by taking $G$ to be a regular graph. In this paper, we show that for every graph $G$, $\text{la}(G)\leq \frac{\Delta}{2}+O(\Delta^{2/3-\alpha})$ for some $\alpha > 0$, thereby improving the previously best known bound due to Alon and Spencer from 1992. For graphs which are sufficiently good spectral expanders, we give even better bounds. Our proofs of these results further give probabilistic polynomial time algorithms for finding such decompositions into linear forests.
\end{abstract}

\section{Introduction}
A \emph{linear forest} is a forest in which every connected component is a path. Given a graph $G$, we define its \emph{linear arboricity}, denoted by $\text{la}(G)$, to be the minimum number of edge-disjoint linear forests in $G$ whose union is $E(G)$. This notion was introduced by Harary \cite{Ha} in 1970 as one of the covering invariants of graphs, and has been studied quite extensively since then.

It is immediate that $\la(G) \leq e(G)$ as every edge $uv$ (along with the isolated vertices $V(G)\setminus\{u,v\}$) forms a linear forest. A less trivial upper bound can be obtained as follows: by a classical theorem due to Vizing, $E(G)$ can be partitioned into at most $\Delta+1$ matchings, where $\Delta:=\Delta(G)$ denotes the maximum degree of $G$; observe that each matching is a linear forest, and therefore we get that $\la(G) \leq \Delta + 1$. For a lower bound, note that every linear forest has at most $n-1$ edges (and equality holds if and only if the linear forest is a Hamiltonian path). Therefore, if $G$ is a $\Delta$-regular graph, then 
$$\la(G) \geq \frac{e(G)}{(n-1)} \geq \frac{n\Delta}{2(n-1)} > \frac{\Delta}{2},$$
which implies (recall that $\la(G)$ is an integer) that $\la(G) \geq \lceil(\Delta+1)/2\rceil$. The following conjecture, known as the \emph{linear arboricity conjecture}, of Akiyama, Exoo and Harary \cite{AEH} asserts that this bound is the best possible:

\begin{conjecture}[The linear arboricity conjecture]
\label{conjecture:LAC}
Let $G$ be a graph of maximum degree $\Delta$.  Then,
$$\text{\emph{la}}(G)\leq \left\lceil \frac{\Delta+1}{2}\right\rceil.$$
\end{conjecture}

\begin{remark}
It is easy to see that every graph $G$, with maximum degree $\Delta(G)$, can be embedded into a $\Delta(G)$-regular graph (perhaps on a greater number of vertices). Therefore, the above conjecture is equivalent to the statement that for a $\Delta$-regular graph $G$ we have $\text{la}(G) =\lceil (\Delta(G)+1)/2 \rceil$.  
\end{remark}

The linear arboricity conjecture was shown to be asymptotically correct as $\Delta \to \infty$ by Alon in 1988 \cite{AlonLA}. He showed that for every $\Delta$-regular graph $G$, $$\la(G) \leq \frac{\Delta}{2} + O\left(\frac{\Delta\log{\log{\Delta}}}{\log{\Delta}}\right);$$ in the same paper, he also proved that the linear arboricity conjecture holds for graphs $G$ with girth $\Omega(\Delta(G))$. The bound for general graphs was subsequently improved by Alon and Spencer in 1992 (see \cite{AlonSpencer}) to: 
\begin{equation}
\label{eq:Noga's bound}
\la(G) \leq \frac{\Delta}{2} + O\left(\Delta^{2/3}(\log{\Delta})^{1/3}\right).
\end{equation}
Even though this conjecture has received a considerable amount of attention over the years, and has been proven (i) in special cases (see, e.g.,  \cite{AEH, AEH2, AlonLA, EP,Guldan,Wu,WW}) (ii) for almost all $d$-regular graphs of constant degree by McDiarmid and Reed \cite{McD}, and (iii) for a typical Erd\H{o}s-Renyi graph with edge-density either $\log^{117}n/n \leq p=o(1)$ or $p$ a fixed constant by Glock, K\"uhn and Osthus \cite{GKO}, there have been no asymptotic improvements in the error term (that is, the second summand in the bound \eqref{eq:Noga's bound} of Alon and Spencer) for general graphs. Our first main result improves this term by a \emph{polynomial} factor: 

\begin{theorem}
\label{main:large d}
There exist absolute constants $\alpha > 0$ and $C>0$ for which the following holds. For any $\Delta$-regular graph $G$, 
$$\la(G) \leq \frac{\Delta}{2} + C{\Delta}^{\frac{2}{3} - \alpha}.$$
\end{theorem}
\begin{remark}
\label{remark: main}
 In the proof of \cref{main:large d}, we make use of \cref{lemma:avoiding short cycles}, the proof of which relies on a `nibbling' argument. As this argument is well-known but quite lengthy, we have used the results from \cite{DGP} as a black box, and we get a bound of (say) $\alpha=1/100$. While a more careful analysis of the nibbling process tailored to our argument may very well give a better bound on $\alpha$, we have made no attempt to do so, since we believe that a `natural barrier' for our argument should be $\alpha = 1/6$ i.e. $\sqrt{\Delta}$ (which is anyway far from \cref{conjecture:LAC}), and any further progress towards the conjecture should require new ideas.  
\end{remark}

It was shown by Peroche \cite{peroche1982complexite} that computing the linear arboricity of a graph is $NP$-complete; this is to be contrasted with variants like the arboricity of a graph (i.e. the minimum number of edge-disjoint forests in $G$ whose union is $E(G)$) for which polynomial time algorithms are available \cite{gabow1992forests}. Our proof of \cref{main:large d} gives an algorithm for computing a decomposition of $E(G)$ into at most $\frac{\Delta}{2} + C{\Delta}^{\frac{2}{3} - \alpha}$ edge-disjoint linear forests, which runs in time polynomial in $|V(G)|$ with high probability. Since the linear arboricity of a $\Delta$-regular graph is at least $\frac{\Delta}{2}$, we thereby get an approximation algorithm providing the best-known approximation guarantee (to our knowledge) for efficiently approximating the linear arboricity of a regular graph.

\begin{corollary}
There exist absolute constants $\alpha>0$ and $C>0$ for which the following holds. Let $G$ be a $\Delta$-regular graph. Then, there is a probabilistic polynomial time algorithm for approximating $\la(G)$ to within $\left(1+\frac{C}{\Delta^{1/3 + \alpha}}\right)$-multiplicative error. 
\end{corollary}



Our second main result deals with $(n,\Delta,\lambda)$-graphs, which we now define. A $\Delta$-regular graph $G$ is said to be an $(n,\Delta,\lambda)$-graph if $|V(G)|=n$ and the second largest (in absolute value) eigenvalue of the adjacency matrix of $G$ is at most $\lambda$. For all such graphs with $\lambda$ not too large compared to $\Delta$, we are able to obtain better bounds on the error than the one coming from \cref{main:large d}. 



\begin{theorem}
\label{main: random regular}
There exist absolute constants $\beta > 0$ and $C > 0$ for which the following holds. For every $(n,\Delta,\lambda)$-graph $G$ with $\lambda \leq \Delta^{2/3}$, 
$$\text{\emph{la}}(G)\leq \frac{\Delta}{2}+C(\lambda\Delta)^{\frac{2}{5}-\beta}.$$
\end{theorem}

Just like for \cref{main:large d}, our proof of \cref{main: random regular} also leads to an algorithm for computing such a decomposition of $E(G)$ in time which is polynomial in $|V(G)|$ with high probability.  

\subsection{The general proof scheme}
\label{sec:outline}
Our proof outlines follow and extend ideas from \cite{AW}. Let $G$ be a $d$-regular graph on $n$ vertices. Consider the following procedure to upper bound $\text{la}(G)$: First, find a vertex partitioning $V(G)=V_1\cup V_2\cup \ldots \cup V_t$, where $t$ is an even positive integer to be specified later, with the following properties:
\begin{enumerate}
  \item $\big||V_i|-|V_j|\big|\leq 1$ for all $i,j\in [t]$, and
  \item $d(v,V_i)\in \frac{d}{t}\pm 100\left(\frac{d\log d}{t}\right)^{1/2}$ for all $v\in V(G)$ and all $i \in [t]$.
\end{enumerate}

The existence of such a partitioning is guaranteed by \cref{lemma: vertex partitioning}, which is proved by a standard application of Chernoff's bounds (\cref{chernoff}) followed by the Lov\'asz Local Lemma (\cref{LLL}).

Second, for all $i\neq j$, let $B_{ij}$ be the induced bipartite graph $G[V_i,V_j]$. By Property $2$ and Vizing's theorem (\cref{thm:vizing}), one can decompose $E(B_{ij})$ into at most 
$$\Delta(B_{ij})+1\leq \bigg\lfloor \frac{d}{t}+102\left(\frac{d\log d}{t}\right)^{1/2}\bigg\rfloor:=s$$ matchings. Let $\mathcal M_{ij}$ be any such decomposition into $s$ matchings (it might be the case that a few of them are empty), and let  $\mathfrak{M}:=\{\mathcal M_{ij}\}_{i<j}$ be the collection of all such decompositions (that is, one decomposition for every $B_{ij}$).

Third, let $\mathcal{P}:=\{P_1,\ldots,P_{t/2}\}$ be a Hamiltonian path decomposition of $K_t$; the existence of such a decomposition is ensured by the fact that $t$ is even and a classical result of Walecki from the 1890s which can be found in \cite{Luc} and provides an explicit such decomposition. It is easy to see that using our collection of decompositions $\mathfrak{M}$, one can find a collection $\mathcal{F_P}$ of forests, one for every such Hamiltonian path $P=v_{i_1},\ldots,v_{i_t}$ in $\mathcal{P}$, satisfying the following two properties: 
\begin{itemize}
  \item $\mathcal F_P$ consists of at most $s$ edge-disjoint linear forests;
  \item $\mathcal F_P$ contains all the edges $\bigcup_{ij \in E(P)} E(B_{ij})$.
\end{itemize}

Indeed, let $P$ be such a Hamiltonian path; after possibly relabeling the vertices, we may assume that $P=123\ldots t$. Observe that by taking one matching from each decomposition $\mathcal M_{i,i+1}$ we obtain a linear forest. Therefore, by repeating this procedure $s$ times, since each $\mathcal M_{i,i+1}$ consists of at most $s$ matchings, one can build a collection of at most $s$ linear forests for every $P$. Clearly, such a collection contains edge-disjoint linear forests whose union consists of all the edges of all the bipartite graphs $\{B_{i,i+1}\}_{i\in[t-1]}$. \\

As there are $\frac{t}{2}$ Hamiltonian paths in $\mathcal{P}$, the above construction gives us at most
$$\frac{st}{2} \leq \frac{d}{2}+ 51\left(dt\log d\right)^{1/2}$$
linear forests which cover all the edges in all the bipartite graphs $\{B_{ij}\}_{i\neq j}$. Let $L:=\bigcup_{i=1}^tE(G[V_i])$ be the set of all the edges which have not been covered by these linear forests (we will also identify $L$ with the graph on $V(G)$ whose edges are $L$, in which case we will refer to $L$ as the \emph{leave graph}). Since $\Delta(L)\leq \frac{d}{t}+100\left(\frac{d\log d}{t}\right)^{1/2}$ by Property $2$ of the partitioning, Vizing's theorem shows that $L$ can be decomposed into at most $\Delta(L)+1 \leq \frac{d}{t} + 101\left(\frac{d\log{d}}{t}\right)^{1/2}$ matchings. Since any matching is manifestly a linear forest, we have thus obtained a decomposition of the edges of $G$ into at most \begin{equation}
\label{eq:la upperbound}
\frac{st}{2}+\Delta(L)+1\leq \frac{d}{2}+\frac{d}{t}+ 152\left(dt\log d\right)^{1/2}
\end{equation}
linear forests. In order to optimize the error term $d/t + 152(dt\log{d})^{1/2}$, we would like to pick $t$ so that the two summands in the error term are the same. This is achieved by choosing $t^{3}=\Theta(d/\log{d})$, in which case 
$$\text{la}(G)\leq \frac{d}{2}+\Theta(d^{2/3}\log^{1/3}d).$$
This is the strategy used in \cite{AW} to recover the bound of Alon and Spencer. \\

Let us now discuss the weak points in the construction and the analysis that we have presented, along with ideas for improving them. The formal details will be given in subsequent sections. 
\begin{enumerate}[(i)]
\item In the above construction, we decompose the leave graph $L$ into matchings and treat each matching as a linear forest by itself. This gives us the $\Theta(d/t)$ error term in the above analysis. Note, however, that adding a matching contained in some $G[V_i]$ to any of the linear forests obtained from a path $P$ which has $i$ as an endpoint still results in a linear forest. Therefore, it makes sense to try to `swallow' all the edges of $L$ in our current linear forests. We discuss this in more detail in \cref{sec:avoiding short cycls in dense}, where we also present the key technical lemma (\cref{lemma:avoiding short cycles}) needed to make this idea work. The upshot of \cref{lemma:avoiding short cycles} is that it allows us to replace the $\Theta(d/t)$ term in the error by $\Theta\left((d/t)^{1-\gamma}\right)$ for some $\gamma > 0$. Optimizing the error term now results in the choice $t^{3-2\gamma}=\Theta(d^{1-2\gamma}/\log{d})$, which gives an error of $O(d^{2/3-\alpha})$ for some $\alpha > 0$, as desired in \cref{main:large d}.       
\item In the above construction, we take $\Delta(B_{ij})+1$ matchings in each $B_{ij}$, whereas ideally, we would like to take only `average degree' many matchings. This error, summed up for all $B_{ij}$, gives us the $\Theta(\sqrt{dt\log d})$ term in \eqref{eq:la upperbound}. In the proof of \cref{main: random regular}, we will show (\cref{r factor}) that if $G$ satisfies some expansion properties, then we can approximately decompose each $B_{ij}$ into edge-disjoint perfect (up to divisibility) matchings. If we remove the linear forests generated by these matchings using the above procedure, then we remove the ``correct'' number of linear forests, and the ``leave graph'' $L$ has much smaller maximum degree. Now, we apply \cref{main:large d} to $L$. 
\end{enumerate}


\section{Auxiliary lemmas}

In this section, we gather various preliminaries, as well as state and prove the key lemmas needed for our proofs.

\subsection{Probabilistic estimates}
Throughout this paper, we will make extensive use of the following well-known bound on the upper and lower tails of a sum of independent indicators, due to Chernoff (see, e.g., Appendix A in \cite{AlonSpencer}).

\begin{lemma}[Chernoff's inequality]
\label{chernoff}	
Let $X_1,\ldots,X_n$ be independent random variables with $\Pr[X_i=1]=p_i$ and $\Pr[X_i=0]=1-p_i$ for all $i$. Let $X=\sum_iX_i$, and let ${\mathbb E}(X) = \mu$. Then
			\begin{itemize}
				\item
				${\mathbb P}[X < (1 - a)\mu ] < e^{-a^2\mu /2}$
				for every $a > 0$;
				\item ${\mathbb P}[X > (1 + a)\mu ] <
				e^{-a^2\mu /3}$ for every $0 < a < 1$.
			\end{itemize}
		\end{lemma}	
\noindent
\begin{remark}\label{rem:hyper} If all the $p_i$'s are the same, the obtained bounds are also valid if, instead of taking $X$ as the sum of i.i.d random variables, we take it to be hypergeometrically distributed with
mean $\mu $ \cite{hoeffding}. 
\end{remark}

Before introducing the next tool to be used, we need the following
definition.

\begin{definition}
  Let $(A_i)_{i=1}^n$ be a collection of events in some probability
space. A graph $\Gamma$ on the vertex set $[n]$ is called a
\emph{dependency graph} for $(A_i)_i$ if $A_i$ is mutually
independent of all the events $\{A_j: ij\notin E(\Gamma)\}$.
\end{definition}

The following is the so-called Lov\'asz local lemma in its symmetric version (see, e.g., \cite{AlonSpencer}).

\begin{lemma}[Lov\'asz local lemma]
\label{LLL}
Let $(A_i)_{i=1}^n$ be a sequence of events in some probability
space, and let $\Gamma$ be a dependency graph for $(A_i)_i$.  Suppose that $\Pr\left[A_i\right]\leq q$ for every $i \in [n]$ and $eq(\Delta(\Gamma)+1)<1$. Then, $\Pr[\bigcap_{i=1}^n \overline{A}_i]>0$.
\end{lemma}

\subsection{Algorithmic Lov\'asz local lemma}
The original proof of the Lov\'asz local lemma in \cite{erdHos1973problems} is non-constructive in that it does not provide any way of finding a point in the probability space avoiding the `bad' events. However, in the case when the `bad' events $(A_i)_{i=1}^{n}$ are determined by a finite collection of mutually independent random variables $(X_j)_{j=1}^{m}$, the breakthrough work of Moser and Tardos \cite{moser2010constructive} shows that the following simple randomized algorithm efficiently computes an assignment to the random variables $(X_j)_{j=1}^{m}$ which avoids all the `bad' events -- start with a random assignment to the variables $\mathcal{P}$, and check whether some event in $\mathcal{A}$ is violated. If so, arbitrarily pick such a violated event, and sample another random assignment for the values of the variables on which this event depends (this step is called a \emph{resampling} of the event). Continue this process until there are no violated events.

\begin{theorem}[\cite{moser2010constructive}]\label{thm:algorithmicLLL} Let $\mathcal{P}$ be a finite set of mutually independent random variables in a probability space. Let $\mathcal{A}$
be a finite set of events determined by these variables. Consider the dependency graph $\Gamma$ on these events given by adding an edge connecting two events $A,B\in \mathcal{A}$ if and only if $A$ and $B$ depend on some common random variable in $\mathcal{P}$. Suppose that $\Pr[A] \leq q$ for every $A\in \mathcal{A}$ and $eq(\Delta(\Gamma)+1) < 1$. Then, there exists an assignment of values to the variables $\mathcal{P}$ not violating any of the events in $\mathcal{A}$. Moreover
the randomized algorithm described above resamples an event $A\in\mathcal{A}$ at most an expected $O(1/\Delta(\Gamma))$ times before it finds such an evaluation. Thus, the expected total number of resampling steps is at most $|\mathcal{A}|/\Delta(\Gamma).$
\end{theorem}

\begin{remark}
All the applications of the local lemma in this paper fit the general framework of the above theorem and seek to avoid at most $p_1(n)$ events, where $p_1(n)$ is some polynomial in the number of vertices $n:=V(G)$. Moreover, every event in each of our applications can be sampled in time $p_2(n)$, for some polynomial $p_2(n)$. It follows that all of our applications of the local lemma can be performed algorithmically in expected time $O(p_1(n)p_2(n))$. Thus, by Markov's inequality, it follows that the probability of the algorithm taking more than $O(p_1(n)p_2(n)p_3(n))$ time is at most $1/p_3(n)$. 
\end{remark}



\subsection{Vizing's theorem}
The \emph{chromatic index} of a graph $G$, denoted by $\chi'(G)$, is the minimum number of colors needed to color $E(G)$ in such a way that each color class is a matching. It follows immediately from this definition that $\chi'(G)\geq \Delta(G)$; perhaps surprisingly, Vizing \cite{Vizing} proved that this trivial lower bound is nearly optimal:

\begin{theorem}[Vizing's Theorem]
  \label{thm:vizing}
Every graph $G$ satisfies
  $$\chi'(G)\in \{\Delta(G),\Delta(G)+1\}.$$
\end{theorem}
Moreover, the strategy in Vizing’s original proof can be used to obtain a
polynomial time algorithm to edge color any graph $G$ with $\Delta(G) + 1$ colors (\cite{misra1992constructive}). 
Note that, as mentioned in the introduction, Vizing's theorem immediately gives the bound $\text{la}(G)\leq \Delta(G)+1$. 

\subsection{Random vertex partitioning}
Given a $d$-regular graph with $d$ sufficiently large, the following lemma gives a partition $V(G)=V_1\cup \ldots \cup V_t$ for which `all the degrees are correct'.

\begin{lemma}
\label{lemma: vertex partitioning}
There exists an absolute constant $d_0$ for which the following holds. For all $d\geq d_0$, all $d$-regular graphs $G$, and all integers $1\leq t\leq d/100$, there exists a partition $V(G)=V_1\cup \ldots \cup V_t$ satisfying the following two properties:
\begin{enumerate}
    \item For all $1\leq i,j\leq t$, $\big| |V_i|-|V_j|\big|\leq 1$ .
    \item For all $v\in V(G)$ and for all $i\in [t]$, the number of edges from $v$ into $V_i$, denoted by $d_G(v,V_i)$, satisfies $d_{G}(v,V_i)\in \frac{d}{t}\pm 100\left(\frac{d\log d}{t}\right)^{1/2}.$
  \end{enumerate}
\end{lemma}

\begin{proof}
Note that for (say) $d\geq \log^2n$, the lemma follows easily by Chernoff's inequality for the hypergeometric distribution and the union bound. Since we are also interested in graphs with smaller degree, we need a slightly more complicated proof where the union bound is replaced by a standard application of the local lemma (\cref{LLL}).   

Let $s:=\lceil n/t\rceil$ and let $F_1,\ldots, F_s$ be an arbitrary partition of $V(G)$ such that $F_1,\dots, F_{s-1}$ are of size $t$ each. 
Let $f:V(G)\rightarrow \{1,\ldots, t\}$ be a random function chosen as follows: for each $k\in[s]$, the restriction $f\vert_{F_k}$ is a permutation of $\big[|F_k|\big]$ chosen uniformly at random. 
Given such an $f$, define $V_i:=\{v\in V(G) : f(v)=i\}$. Observe that for each $i \in [t]$, $|V_i|$ is either $s-1$ or $s$, so that the desired property $1$ of the lemma holds. We wish to show that, with positive probability, there exists an $f$ such that the corresponding partition $V(G)=V_1\cup \dots \cup V_t$ satisfies property $2$ of the lemma.
  
To this end, fix a vertex $v \in V(G)$ and for each $k\in [s]$, let $p_k:=|N_G(v)\cap F_k|/|F_k|$. Since each $f\vert_{F_k}$ is chosen uniformly at random from among all permutations of $\big[|F_k|\big]$, it follows that for all $i \in [t]$, 
$$\frac{d}{t}-1\leq \sum_{i=1}^{s-1}p_i \leq \mathbb{E}[d_{G}(v,V_i)]= \sum_{i=1}^{s}p_i\leq \frac{d}{t} + 1.$$ 
Therefore, by Chernoff's bounds (\cref{chernoff}), 
$$\Pr\left[d_G(v,V_i)\notin \frac{d}{t}\pm 100\sqrt{\frac{d\log d}{t}}\right]\leq \exp\left(-2000\log d\right)=\frac{1}{d^{2000}}.$$
  
Let $\mathcal E_{i,v}$ denote the event `$d(v,V_i)\notin \frac{d}{t}\pm 100\sqrt{\frac{d\log d}{t}}$', and note that for all $i\in [t]$ and $v\in V(G)$, $\mathcal E_{i,v}$ may depend on an event $\mathcal E_{j,u}$ only if at least one of the following two conditions hold: $u=v$; or $u$ and $v$ have neighbors to the same $F_k$ for some $k$. In particular, each event $\mathcal E_{i,v}$ depends on at most $t+d^2t\leq d^{4}$ events. Finally, since 
  $$\frac{e(d^4+1)}{d^{2000}}<1,$$
the local lemma guarantees the existence of an $f$ as desired. 
\end{proof}

\subsection{Finding dense, regular spanning subgraphs in `nice' bipartite graphs}
The next lemma shows that almost-regular balanced bipartite graphs induced by large disjoint subsets of a good expander contain a spanning regular graph covering almost all the edges. The proof is similar to the proof of Lemma 2.12 in \cite{FJ}, and is based on the following generalization of the Gale-Ryser theorem due to Mirsky \cite{Mirsky}.
\begin{theorem}[\cite{Mirsky}]
  \label{gale}
Let $G=(A\cup B,E)$ be a balanced bipartite graph with $|A|=|B|=m$, and let $r$ be an integer. Then, $G$ contains an $r$-factor if and only if for all $X\subseteq A$ and $Y\subseteq B$
\begin{align*}
  e_G(X,Y)\geq r(|X|+|Y|-m).
\end{align*}
\end{theorem}
\begin{lemma} 
\label{r factor}
Let $G$ be an $(n,d,\lambda)$ graph. Let $1\leq t\leq d/100$ be some integer such that $t\mid n$. Let $A$ and $B$ be disjoint subsets of $V(G)$ of sizes $|A|=|B|=\frac{n}{t}$ and consider the bipartite subgraph of $G$ induced by these sets, denoted by $G':=G[A,B]$. Assume further that $\frac{d}{t} - 100\sqrt{\frac{d\log{d}}{t}}\leq \delta(G')\leq \Delta(G') \leq \frac{d}{t} + 100\sqrt{\frac{d\log{d}}{t}}$. Then, $G'$ contains an $r:=\lfloor\frac{d}{t}-\gamma\rfloor$-factor (i.e. an $r$-regular spanning subgraph) for $\gamma = 104\max\{\lambda, \sqrt{\frac{d\log{d}}{t}}\}$, provided that $\gamma < r/2$.  
\end{lemma}
\begin{proof}
Since $\gamma \geq 0$, the statement is vacuously true whenever $r\leq 0$. Hence, we may assume that $r > 0$. By Mirsky's criterion, it suffices to verify
that for all $X\subseteq A$ and $Y\subseteq B$, we have
$$e_{G}(X,Y)\geq \left(\frac{d}{t}-\gamma\right)\left(|X|+|Y|-\frac{n}{t}\right).$$
We divide the analysis into five cases:

{\bf Case 1:} $|X|+|Y|\leq \frac{n}{t}$. Since $e_G(X,Y)\geq 0$, there is nothing to prove in this case.

{\bf Case 2:} $|X|+|Y| > \frac{n}{t}$, $|Y| \geq |X|$ and $|X||Y^{c}| \leq \frac{\lambda^{2} n^{2}}{d^{2}}$, where $Y^{c}:= B\setminus Y$. Suppose for contradiction that $e_{G}(X,Y)<r\left(|X|+|Y|-\frac{n}{t}\right).$
Then, it must be the case that 
\[
e_{G}(X,Y^{c})\geq\delta(G')|X|-e_{G}(X,Y)>(\delta(G')-r)|X|+r|Y^{c}|.
\]

On the other hand, we know by the expander mixing lemma that 
\[
e_{G}(X,Y^{c})\leq\frac{d}{n}|X||Y^{c}|+\lambda\sqrt{|X||Y^{c}|}\leq2\lambda\sqrt{|X||Y^{c}|},
\]

where the second inequality holds since $|X||Y^{c}|\leq\frac{\lambda^{2}n^{2}}{d^{2}}$.
Hence, we must have 
\[
2\lambda\sqrt{|X||Y^{c}|}\geq(\delta(G')-r)|X|+r|Y^{c}|.
\]

Since both terms on the right hand side are nonnegative, $2\lambda\sqrt{|X||Y^{c}|}$
should also be greater than either of them, for which we must have
\[
\frac{r^{2}}{4\lambda^{2}}|Y^{c}|\leq|X|\leq\frac{4\lambda^{2}}{(\delta(G')-r)^{2}}|Y^{c}|.
\]

In particular, we must have $r(\delta(G')-r)\leq4\lambda^{2}$, which implies $\lambda \geq r$,
which violates our assumption about $\lambda$. 
 
{\bf Case 3:} $|X|+|Y|> \frac{n}{t}$, $|Y|\geq |X|$, $|X||Y^c| > \frac{\lambda^{2} n^{2}}{d^{2}}$ and $|X| < \frac{nr}{2d}$. If $e_G(X,Y) < r\left(|X| + |Y| - \frac{n}{t}\right)$, then by the same argument as above, we must have 
$$2\frac{d}{n}|X||Y^c| \geq (\delta(G')-r)|X| + r|Y^c|.$$
In particular, we must have $|X| \geq \frac{nr}{2d}$, which violates our assumption about $|X|$. 

{\bf Case 4:} $|X| + |Y| > \frac{n}{t}$, $|Y| \geq |X|$ and $|X| \geq \frac{nr}{2d}$. By assumption, we have $2|Y| > n/t$, so that $|Y| > n/2t$. Moreover, since $\gamma < r/2$, we have that $r > d/2t$. Therefore, $\sqrt{|X||Y|} \geq \sqrt{n^{2}/4t^{2}} \geq n/2t$. On the other hand, we also have $(2n\lambda)/(t\gamma ) \leq (2n\gamma) /(104 t\gamma ) \leq n/50t$. Combining these two inequalities, we see that $\frac{\gamma}{2}|X||Y| \geq \frac{n\lambda}{t}\sqrt{|X||Y|}$. Therefore, by the expander mixing lemma, it suffices to verify that 
\[
\frac{t}{n}\left(\frac{d}{t}-\frac{\gamma}{2}\right)|X||Y|\ge\left(\frac{d}{t}-\gamma\right)\left(|X|+|Y|-\frac{n}{t}\right).
\]
 Dividing both sides by $n/t$, we see that this is implied by the inequality 
$$xy-\beta(x+y-1)\geq0,$$
where $\beta = \frac{d/t-\gamma}{d/t-\gamma/2}$, $x=t|X|/n$, $y=t|Y|/n$,  $x+y\geq1$, $0\leq x\leq1$, and $0\leq y\leq1$. Observe that the objective function on the left hand side of the desired inequality is bilinear in $x$ and $y$, and therefore the minimum will be obtained on the triangular boundary of the region. On this boundary, the inequality reduces to one of the following:  $xy\geq 0$; $x \geq \beta x$; $y\geq \beta y$, and is readily verified since $\beta \leq 1$. 

{\bf Case 5:} $|X|+|Y|>\frac{n}{t}$ and $|Y|\leq |X|$. This is
exactly the same as cases (2)-(4) with the roles of $X$ and $Y$ interchanged.
\end{proof}

\begin{remark}
Under the conditions of the above lemma, an $r$-factor in $G'$ can be found efficiently using algorithmic versions of Mirsky's criterion based on standard network flow algorithms (see, e.g., \cite{anstee1983network}). 
\end{remark}

\begin{remark} 
\label{rmk:divisibility}
In the application of this lemma to \cref{main: random regular}, we will have to deal with bipartite graphs as above, except that we are allowed to have $|A| = |B| + 1$. In this scenario, it is impossible to find an $r$-factor. However, by adding a ``fake'' vertex to $B$ with suitable edge connections to $A$, finding an $r$-factor in this new graph using the above lemma, decomposing this $r$-factor into $r$ edge-disjoint perfect matchings using repeated applications of Hall's theorem, and finally removing all edges incident to the ``fake'' vertex, we see that $G':=G[A,B]$ contains $r$-edge disjoint matchings such that every vertex in $A\cup B$ is matched in at least $r-1$ such matchings. 
\end{remark}

\subsection{Avoiding short cycles}
\label{sec:avoiding short cycls in dense}

In this section, we introduce our key technical lemma for proving \cref{main:large d}. Since the usefulness of this lemma may not be apparent at first glance, we encourage the reader to refer to this section only after encountering its application in the proof of \cref{main:large d}.  

\begin{lemma}
  \label{lemma:avoiding short cycles}
There exist universal constants $B, D > 20$ for which the following holds. Let $G$ be a graph with maximum degree $\Delta$ and minimum degree $\delta$ such that $\Delta - \delta \leq \Delta^{5/6}$ and $\Delta \geq D$. Let $M_1,\ldots,M_{\Delta+1}$ be a fixed collection of matchings in the complete graph on $V(G)$. Then, there exists a collection of matchings $M'_1,\ldots, M'_{\Delta+1}$ in $G$, where some of them may possibly be empty, such that:
\begin{enumerate}
\item the graph $G'$, which is obtained from $G$ by deleting all the edges $\bigcup_{i\in[\Delta+1]}E(M'_i)$, has maximum degree at most $\Delta^{1-1/B}$;
\item for all $v\in V(G)$, there are at most $\Delta^{b}$ indices $i\in [\Delta+1]$ for which $v$ lies on a cycle in $M'_i \cup M_i$ of length at most $\Delta^{1/40}$, where $b:=\frac{6}{B} + \frac{1}{20}$. 
\end{enumerate}
Moreover, such a collection of matchings may be obtained in $\text{poly(V(G))}$ time with high probability. 
\end{lemma}

The proof of this lemma builds on the proof of the main result in the work of Dubhashi, Grable, and Panconesi \cite{DGP}. Since the details are somewhat involved, we defer them to \cref{appendix:nibbling}.

\section{Proofs of main results}
\label{sec:proofs}
In this section, we conclude the proofs of our main results. Since these proofs build on the general strategy discussed earlier, we encourage the reader to review the construction in \cref{sec:outline} before proceeding. We start by proving \cref{main: random regular} as a warm up since its proof is simpler. 

\subsection{Proof of \cref{main: random regular}}
Let $G$ be an $(n,d,\lambda)$-graph and set $\gamma=104\max\{\lambda,\sqrt{\frac{d\log d}{t}}\}$.  As in the general proof scheme presented in \cref{sec:outline}, we start with a vertex decomposition $V(G)=V_1\cup\ldots V_t$ satisfying the conclusions of \cref{lemma: vertex partitioning}, where $t \leq d/100$ is a positive even integer which will be specified below. For all $i\neq j$, let $\mathcal{M}_{ij}$ be a collection of $r:=\lfloor \frac{d}{t}-\gamma\rfloor$ edge-disjoint matchings of the bipartite graph $B_{ij}:=G[V_i,V_j]$ as in \cref{rmk:divisibility} -- such a decomposition exists for all sufficiently large $d$ since $\gamma < \frac{r}{2}$ holds by our choice of $t$ below, and our assumption that $\lambda \leq d^{2/3}$.  

Let $\mathcal P:=\{P_1,\ldots,P_{t/2}\}$ be a Hamiltonian path decomposition of $K_t$, and for each $P\in \mathcal P$, let $\mathcal F_P$ be the collection of $r$ edge-disjoint linear forests obtained as in \cref{sec:outline}. This gives us a set of $\frac{rt}{2}$ edge-disjoint linear forests of $G$. The key observation here is that the graph $L$ induced by all edges of $G$ which are not in any such linear forest has 
maximum degree $\Delta(L) \leq d-(t-1)(r-1) \leq (\gamma+1)t + r$ since each vertex in $V_i\cup V_j$ is in at least $r-1$ of the $r$ edge-disjoint matchings selected in $B_{ij}$. Our goal now is to find a decomposition of the edges of $L$ into as few linear forests as possible.  
The bound \eqref{eq:Noga's bound} ensures that we can find a decomposition into at most $
\Delta(L)/2+ C(\Delta(L))^{2/3}\log ^{1/3}(\Delta(L))$ linear forests. Together with the collection of $rt/2$ edge-disjoint linear forests that we built earlier, this shows that
\begin{align*}
\text{la}(G)
&\leq \frac{rt}{2} + \frac{\Delta(L)}{2} + C(\Delta(L))^{2/3}\log ^{1/3}(\Delta(L))\\
&\leq \frac{rt}{2} + \frac{\gamma t + t + r}{2} + C(\Delta(L))^{2/3}\log ^{1/3}(\Delta(L))\\
&\leq \frac{(\lfloor \frac{d}{t} - \gamma \rfloor +\gamma)t}{2} + \frac{t+r}{2} + C(\Delta(L))^{2/3}\log ^{1/3}(\Delta(L))\\
&\leq \frac{d}{t} + \frac{t+r}{2} + C(\Delta(L))^{2/3}\log ^{1/3}(\Delta(L)).
\end{align*}
Setting $t = \tilde{\Theta}\left(\frac{d^{3}}{\gamma^{2}}\right)^{1/5} $ to optimize the error term (in which case $\gamma = 104\lambda$) shows that 
$\la(G) \leq \frac{d}{2} + \tilde{O}\left((d\lambda)^{2/5}\right)$, where the tilde hides logarithmic dependence on $d$. If instead of \eqref{eq:Noga's bound}, we use \cref{main:large d} to handle the linear arboricity of $L$, then we get that 
$$\la(G) \leq \frac{d}{2} + O\left((d \lambda)^{\frac{2}{5} - \beta}\right) $$
for some $\beta > 0$, as desired. 


\subsection{Proof of \cref{main:large d}}
Let $G$ be a $d$-regular graph on $n$ vertices with $d$ sufficiently large. Let $V(G)=V_1\cup\ldots V_t$ be a vertex-partition satisfying the conclusions of \cref{lemma: vertex partitioning}, where $t \leq d/100$ is a positive even integer which will be specified below. As before, let $\mathcal{M}_{ij}$ denote a decomposition of the bipartite graph $B_{ij}:=G[V_i,V_j]$ into at most $\Delta(B_{ij})+1$ matchings, and let $\mathfrak{M}=\{\mathcal M_{ij}\}_{i\neq j}$ denote the collection of such decompositions.

Let $\mathcal P:=\{P_1,\ldots,P_{t/2}\}$ be a Hamiltonian path decomposition of $K_t$, and for each $P\in \mathcal P$, let $\mathcal F_P$ be the collection of at most $s$ edge-disjoint linear forests obtained as in \cref{sec:outline}. 
Fix an arbitrary labeling $\mathcal F_P = \{F_{P,1},\dots, F_{P,s}\}$ of these forests. Moreover, for each $P\in \mathcal P$, let $s_P$ and $t_P$ denote its endpoints, and observe that all the pairs $\{s_P,t_P\}_{P\in \mathcal P}$ are disjoint.

Next, for each $P\in \mathcal P$, let $\mathcal{M}_{s_{P}}:=\{M_{P,1}'',\dots,M_{P,s}''\}$ be a decomposition of the edges of $G[V_{s_P}]$ into $s$ matchings; the existence of such a decomposition is guaranteed by Vizing's theorem. For each $i \in [s]$, let $F''_{P,i}:=F_{P,i}\cup M''_{P,i}$, and observe that $\mathcal{F''_P}:= \{F''_{P,1},\dots, F''_{P,s}\}$ is a collection of edge-disjoint linear forests which covers all the edges $\bigcup_{ij \in E(P)}E(B_{ij}) \cup E(G[V_{s_P}])$. For each $i\in [s]$, let $M_{P,i}$ be the set of all pairs $\{x,y\}\subseteq V_{t_P}$ for which there exists a path in $F''_{P,i}$ of length exactly $2t-1$ with $x$ and $y$ as its endpoints. Note that such paths correspond precisely to two `full paths' of length $t-1$ in $F_{P,i}$ whose endpoints in $V_{s_P}$ are an edge of $M''_{P,i}$. Since each $M''_{P,i}$ is a matching, it follows immediately that each $M_{P,i}$ is a matching of the complete graph on the vertex set $V_{t_P}$.

For each $P\in \mathcal P$, consider the graph $G^*_P:=G[V_{t_P}]$. By \cref{lemma: vertex partitioning}, we have 
$$\Delta_P-k\leq \delta(G^*_P)\leq \Delta(G^*_P)=:\Delta_P,$$
where $k = 200((d\log{d})/t)^{1/2}$. Below, we will choose $t$ be to less than $\sqrt{d}$. Therefore, for $d$ sufficiently large, $\Delta_{P} \geq D$ and $k \leq \Delta_{P}^{5/6}$, so that by applying \cref{lemma:avoiding short cycles} to $G^*_P$ we obtain a collection of $\Delta_P + 1$ matchings $M'_1,\ldots,M'_{\Delta_P + 1}$ in $G^*_P$, where some of them are possibly empty, such that:

\begin{enumerate}
  \item the graph $G'_P$, which is obtained from $G^*_P$ by deleting all the edges $\bigcup_{i\in [\Delta_P + 1]}E(M'_{P,i})$,  has maximum degree at most $\Delta_P^{1-1/B}$, and
  \item for all $v\in V(G)$, there are at most $\Delta_P^{b}$ indices $i\in [\Delta_P+1]$ for which $v$ lies on a cycle in $M_i \cup M'_i$ of length at most $\Delta_P^{1/40}$.
\end{enumerate}

With this in hand, let $F'_{P,i}:= F''_{P,i}\cup M'_{P,i}$ for all $i\in [\Delta_P+1]$, and let $\mathcal F'_P:=\{F'_{P,1},\dots,F'_{P,\Delta_P + 1}\}$. Since each $F'_{P,i}$ is a graph of maximum degree at most $2$, it is a disjoint union of cycles, paths and isolated vertices. We wish to remove one edge from each cycle in each $F'_{P,i}$. For the analysis, it will be convenient to do it in the following manner: for any cycle $C$ in any $F'_{P,i}$ of length at most $t\Delta_P^{1/40}$, remove an edge arbitrarily from $M'_{P,i}$; on the other hand, for a cycle $C$ in some $F'_{P,i}$ of length at least $t\Delta_P^{1/40}$, delete an edge chosen uniformly at random from among the first (with respect to a fixed, but otherwise arbitrary ordering of the edges) $\lfloor \Delta_P^{1/40}/2\rfloor $ edges of $M'_i$ appearing in this cycle. Let $F^*_{P,i}$ denote the (random) linear forest resulting from $F_{P,i}$ after this deletion, and let $\mathcal F^*_P:=\{F^*_{P,1},\dots,F^*_{P,\Delta_P+1}\}$ be the collection of edge-disjoint linear forests obtained from the Hamiltonian path $P$ in this manner. 

For each $v \in V_{t_P}$, let $X(v)$ denote the (random) number of edges in $\bigcup_{i\in [\Delta_P+1]}F^*_{P,i} \setminus F'_{P,i}$ which are incident to $v$. We claim that there is a choice of $\mathcal F^*_P$ for which $X(v) \leq 9\Delta_P^{39/40}$ for all $v \in V_{t_P}$. For this, fix $v\in V_{t_P}$ and observe that since $v$ is part of at most $\Delta_P^{b}$ cycles of length $\leq \Delta_P^{1/40}$ in $M_i\cup M'_i$, it follows that $v$ can be a part of at most $\Delta_P^{b}$ cycles in $F'_{P,i}$ of length $\leq t\Delta_P^{1/40}$. Hence, the contribution to $X(v)$ from such cycles is at most $\Delta_P^{b}$. Moreover, the probability that any cycle of length at least $t\Delta_P^{1/40}$ contributes to $X(v)$ is bounded above by $1/\lfloor \Delta_P^{1/40}/2\rfloor \leq 3/\Delta_P^{1/40}$, since such a cycle contributes to $X(v)$ only when the edge deleted from it is incident to $v$, where the edge to be deleted is chosen uniformly at random from among $\lfloor \Delta_P^{1/40}/2\rfloor$ edges, of which at most one is incident to $v$. Since there are at most $\Delta_P+1$ cycles containing $v$ to start with, and since deletions from long cycles are made independently, it follows from Chernoff's bounds that with probability at least $1-\exp\left(-\Delta_P^{38/40}\right)$, $X(v) \leq \Delta_P^{b} + 8\Delta_P^{39/40} \leq 9\Delta_P^{39/40}$. Let $\mathcal E_v$ denote the event that this does not happen. Note that $\mathcal E_v$ can depend on $\mathcal E_u$ only if $v$ and $u$ are both incident to the first $\lfloor \Delta_P^{1/40}/2\rfloor$ edges of $M'_i$ in a long cycle. Again, since there are at most $\Delta_P+1$ cycles to start with, it follows that any $\mathcal E_v$ can depend on at most $\Delta_P^{2}$ other $\mathcal E_u$'s. Therefore, since $\Delta_P^{2}\exp(-\Delta_P^{39/40}) \ll 1$, it follows from the local lemma that $\Pr[\cap_{v\in V_{t_P}}\overline{\mathcal E_v}] > 0$, which proves the desired claim. \\

Finally, repeat the above construction for each $P\in \mathcal P$ to obtain a collection of edge disjoint linear forests $\mathfrak F:= \{\mathcal F^*_P\}_{P\in \mathcal P} $, and let $L$ denote the leave graph obtained by deleting from $G$ any edge which appears in $\mathfrak F$. Observe that $L$ consists of edges of the following two types: 
\begin{itemize}
\item edges within $V_{t_P}$ that are not contained in $\cup_{i\in [\Delta_P+1]}E(M'_{P,i})$ i.e. edges in the graph $G'_P$; 
\item edges in $\cup_{P\in P}F'_P$ that are not contained in $\mathfrak{F}$ i.e. edges removed during the deletion process described above.  
\end{itemize}

Recall from \cref{lemma: vertex partitioning} that $\Delta_P \leq s$ for all $P\in \mathcal{P}$. Since the $V_i's$ are disjoint, it follows from the above discussion that $\Delta(L)\leq 9\Delta_P^{39/40}+\Delta_P^{1-1/B}\leq 10s^{1-\gamma}$, where $\gamma:=\min\{\frac{1}{40},\frac{1}{B}\}$ Therefore, by Vizing's theorem, one can decompose $L$ into at most $10s^{1-\gamma}+1$ edge-disjoint matchings. These matchings, together with $\mathfrak{F}$, give a decomposition of $E(G)$ into a number of linear forests which is at most 
$$\frac{st}{2}+10s^{1-\gamma}+1\leq \frac{d}{2}+200\left(\sqrt{dt\log d}+\left(\frac{d}{t}\right)^{1-\gamma}\right).$$

Optimizing the error term by setting the two summands in the parentheses to be equal gives $t = \left(\frac{d^{1-2\gamma}}{\log d}\right)^{\frac{1}{3-2\gamma}}$,
in which case, we get that 
$$\text{la}(G)\leq \frac{d}{2}+d^{2/3-\alpha},$$
for some $\alpha>0$, as desired. 

\bibliographystyle{abbrv}
\bibliography{arboricity}

\begin{appendix} 
\section{Proof of \cref{lemma:avoiding short cycles}}
\label{appendix:nibbling}
In this appendix, we show how the proof of the main result in \cite{DGP}, which is based on the celebrated R\"odl nibble \cite{Rodl}, implies \cref{lemma:avoiding short cycles}. The organization of this appendix is as follows: \cref{algorithm:nibble} records the nibbling algorithm used in \cite{DGP}; \cref{thm:nibbling-expectation} and \cref{thm:nibbling-variance-bounds} record the conclusion of the analysis in \cite{DGP}; \cref{corollary:nibbling-concentration} adapts the analysis in \cite{DGP} for our choice of parameters; \cref{prop:cycle-bound} and \cref{lemma:one-step} show that \cref{algorithm:nibble} produces only a small number of short cycles with respect to any fixed collection of matchings, and finally, \cref{prop:good-nibbles} proves \cref{lemma:avoiding short cycles}. 

Before proceeding to formal details, let us provide a high level overview of what follows. The goal in \cite{DGP} is to produce a proper edge-coloring of a $\Delta$-regular graph $G$ using $(1+\epsilon)\Delta$ colors (here, $\epsilon$ is allowed to depend on $\Delta$). Their algorithm runs in two phases -- the first phase, which is based on the semi-random `nibble' method of R\"odl, is the one relevant to our paper; the second phase actually uses a trivial algorithm.  In the first phase, the algorithm seeks to color `most' of the edges using a \emph{palette} of $\Delta$ colors. Starting with the input graph $G_0:=G$, the algorithm generates a sequence $G_0,G_1,\ldots, G_{t_\epsilon}$ of graphs, where $G_i$ is the graph induced by the edges which are still uncolored at the end of stage $i$. In each stage $i$, each edge has a palette of all `available' colors, where initially, the palette of each edges is the set $\{1,\ldots,\Delta\}$. 
Each vertex selects an $\epsilon/2$-fraction of uncolored edges incident to it, and each selected edge picks a \emph{tentative} color from its palette independently and uniformly at random. If a \emph{selected} edge has no `color-conflicts' with any neighboring edge, then the corresponding color becomes the final color of the edge. All the palettes of the remaining edges are updated by deleting all the final colors of neighboring colored edges. This process is then repeated in the next stage. The algorithm continues for a number of rounds by the end of which (with high probability) each vertex has no more than $\epsilon \Delta$ uncolored edges incident to it.  

As in all nibbling-based arguments, the key idea is that in each stage, the number of edges which experience color conflicts is only a small fraction of the number of edges selected to be colored at this stage. The main effort in \cite{DGP} is spent in showing that this holds true with high probability throughout the process. They do this by showing inductively -- and this is what we will use in our analysis -- that the graphs $G_i$ and the color palettes of each edge behave almost like `random' subgraphs and subsets of the original ones. We now give a formal description of the algorithm and analysis in \cite{DGP}. Following this, we will show how to tailor it to our application. \\



\cref{algorithm:nibble} is the first phase of the algorithm used in \cite{DGP} as described above.
\begin{algorithm}
\caption{The Nibble Algorithm}
\label{algorithm:nibble}
The initial graph $G_0:= G$, the input graph. 
Each edge $e = uv$ is initially given the palette $A_0(e) = \{1,\dots, \Delta\}$.
For $i=0,\dots,t_\epsilon - 1$ stages, repeat the following:
\begin{itemize}
\item \emph{(Select nibble)} Each vertex $u$ randomly selects an $\epsilon/2$ fraction of the uncolored edges incident to itself. An edge is considered selected if either or both of its endpoints selects it. 
\item \emph{(Choose tentative color)} Each selected edge $e$ chooses independently at random a tentative color $t(e)$ from its palette $A_i(e)$ of currently available colors.
\item \emph{(Check color conflicts)} Color $t(e)$ becomes the final color of $e$ unless some edge incident to $e$ has chosen the same tentative color. 
\item \emph{(Update graph and palettes)} The graph and the palettes are updated by setting 
$$ G_{i+1} = G_i - \{e \vert e \text{ got a final color}\}$$
and, for each edge $e$, setting 
$$ A_{i+1}(e) = A_i(e) - \{t(f)\vert f \text{ incident to } e, t(f) \text{ is the final color of } f \}.$$
\end{itemize}
\end{algorithm}
The analysis of this algorithm is based on controlling the following three quantities:
\begin{itemize}
\item $|A_i(u)|$, the size of the implicit palette of vertex $u$ at the end of stage $i$, where the implicit palette $A_i(u)$ denotes the set of colors not yet successfully used by any edge incident to $u$. 
\item $|A_i(e)|$, the size of the palette $A_i(e)$ of edge $e$ at the end of stage $i$. Note that $A_i(uv) = A_i(u) \cap A_i(v)$. 
\item $\deg_{i,\gamma}(u)$, the number of neighbors of $u$ which, at the end of stage $i$, have color $\gamma$ in their palettes. 
\end{itemize}

Before we discuss their analysis of this algorithm, we need some notation. Define $d_i$ and $a_i$ as follows: first, define initial values
$$d_0, a_0 := \Delta$$
and then, recursively define 
\begin{align*}
d_{i} & :=(1-p_{\epsilon})d_{i-1}=(1-p_{\epsilon})^{i}\Delta;\\
a_{i} & :=(1-p_{\epsilon})^{2}a_{i-1}=(1-p_{\epsilon})^{2i}\Delta=d_{i}^{2}/\Delta,
\end{align*}
where 
\[
p_{\epsilon}:=\epsilon\left(1-\frac{\epsilon}{4}\right)e^{-2\epsilon(1-\epsilon/4)}.
\]
In particular, note that setting 
$$t_\epsilon := \frac{1}{p_\epsilon}\log{\frac{4}{\epsilon}},$$ 
we have $d_{t_\epsilon} \leq {\epsilon\Delta}/4$. Also, provided that $\epsilon < 1/100$, we have $d_{t_\epsilon} = (1-p_\epsilon)^{t_\epsilon} \Delta \geq e^{-2p_\epsilon t_\epsilon}\Delta = \epsilon^{2}\Delta/16$.  

\begin{theorem}[\cite{DGP}, Lemmas 9, 12 and 15, and the discussion in Section 5.5]
\label{thm:nibbling-expectation}
There exist constants $K,D> 0$ such that if $\epsilon < 1/100$, $\epsilon^{2}\Delta \geq D$, and at the end of stage $i$ of \cref{algorithm:nibble} the following holds for all vertices $u$, edges $e$ and colors $\gamma$ with $e_i \leq 1/2$: 
\begin{align*}
|A_{i}(u)| & =(1\pm e_{i})d_{i}\\
|A_{i}(e)| & =(1\pm e_{i})a_{i}\\
\deg_{i,\gamma}(u) & =(1\pm e_{i})a_{i},
\end{align*}
then the following holds for all vertices $u$, edges $e$, and colors $\gamma$:
\begin{align*}
\Ex\big[|A_{i+1}(u)|\big] & =(1\pm (1+K\epsilon)e_{i})d_{i+1}\\
\Ex\big[|A_{i+1}(e)|\big] & =(1\pm (1+K\epsilon)e_{i})a_{i+1}\\
\Ex\big[\deg_{i+1,\gamma}(u)\big] & =(1\pm (1+K\epsilon)e_{i})a_{i+1}.
\end{align*}
\end{theorem}

\begin{remark}
\label{rmk:init-error}
In our case (\cref{lemma:avoiding short cycles}), we have $$|A_0(u)| =|A_0(e)| = \Delta,$$ 
and 
$$\deg_{0,\gamma} = \deg(u) = \Delta\left(1 \pm \frac{\Delta - \delta}{\Delta}\right)$$
for all vertices $u$, edges $e$, and colors $\gamma$. Therefore, we can take $$ e_0 = \frac{\Delta - \delta}{\Delta}\leq \Delta^{-1/6}.$$
\end{remark}

In order to show that the above random variables concentrate around their expectation, we will (as in \cite{DGP}) use the following concentration inequality due to Grable \cite{Grable}. The statement of this inequality uses the notion of the `variance of a strategy for determining a random variable' of the form $Y=f(X_1,\dots,X_n)$, whose definition we reproduce verbatim from \cite{DGP} for the reader's convenience. A \emph{querying strategy} for $Y$ is a decision tree whose internal nodes designate queries to be made. Each node of the tree represents a query
of the type ``what was the random choice of $X_i$?''. A node has as many children as there are random choices for $X_i$. Every path from the root to a node which goes through vertices corresponding to $X_{i_1},\dots,X_{i_k}$ defines an assignment $a_1,\dots,a_k$ to these random variables. We can think of each node as storing the value $\E[Y|X_{i_1}=a_1,\dots,X_{i_k}=a_k]$. In particular, the leaves store the possible values of $Y$, since by then all relevant random choices have been determined.   
Define the \emph{variance of a query} (internal node) $q$ concerning choice $X_i$ to be
$$v_q = \sum_{a\in A_i}\Pr[X_i=a]\mu_{q,a}^{2},$$
where
$$\mu_{q,a} = \E[Y |X_i = a \text{ and all previous queries}] - \E[Y|\text{ all previous queries}].$$
By ``all previous queries'', we mean the condition imposed by the queried choices
and exposed values determined by the path from the root of the strategy down to the
node $q$. In words, $\mu_{q,a}$ measures the amount which our expectation changes when the answer to query $q$ is revealed to be $a$. 
Also define the maximum effect of query $q$ as
$$c_q = \max_{a,b\in A_i}|\mu_{q,a} - \mu_{q,b}|.$$
A way to think about $c_q$ is the following. Consider the children of node $q$; $c_q$ is the maximum difference between any values $\E[Y| \text{ all previous queries}]$ stored at the children. 
A \emph{line of questioning $\ell$} is a path in the decision tree from the root to a leaf and the variance of a line of questioning is the sum of the variances of the queries along it. Finally, the \emph{variance of a strategy} $\mathscr{S}$ is the maximum variance over all lines of questioning
$$V(\mathscr{S}) = \max_\ell \sum_{q\in \ell}v_{q}.$$ 

We are now ready to state Grable's concentration inequality.  
\begin{theorem}[\cite{Grable}] 
\label{grable:concentration-inequality}
Let $\mathscr{S}$ be a strategy for determining a random variable $Y$, and suppose the variance of $\mathscr{S}$ is at most $V$. Then, for every $0\leq \varphi \leq V/\max{c_q^{2}}$,
$$\Pr\left[|Y-\Ex[Y]| > 2\sqrt{\varphi V}\right] \leq 2\exp(-\varphi).$$
\end{theorem}
We note that $\max c_q^{2}$ will always be at most $16$ in all the applications of \cref{grable:concentration-inequality} that we will need (\cite{DGP}).\\

The next theorem records the bounds on the variance (in the above sense) of various random variables which we are interested in. 

\begin{theorem}[\cite{DGP}, Lemmas 10, 13 and 16, and the discussion in Section 5.5.]
\label{thm:nibbling-variance-bounds}
Fix $0\leq i \leq t_\epsilon -1$. Let $X_{u} := |A_{i}(u)| - |A_{i+1}(u)|$, $Y_{e} := |A_i(e)| - |A_{i+1}(e)|$, and $Z_{\gamma,u}: = \deg_{i+1,\gamma}(u)$. 
Suppose the assumptions of \cref{thm:nibbling-expectation} are satisfied. Then, for any fixed $u,e,\gamma$, with probability at least $1 - \exp(-\epsilon d_i/10)$, there exist strategies $\mathscr{S}_{X,u}$, $\mathscr{S}_{Y,e}$ and $\mathscr{S}_{Z,\gamma,u}$ for determining $X_{u}$, $Y_{e}$ and $Z_{\gamma,u}$ with the following variance bounds:
\begin{align*}
V(\mathscr{S}_{X,u}) &\leq 100\epsilon d_i \\
V(\mathscr{S}_{Y,e}) &\leq 10000 a_i \\ 
V(\mathscr{S}_{Z,\gamma,u}) &\leq 100 \epsilon a_i. 
\end{align*}
\end{theorem}

As an immediate application of \cref{grable:concentration-inequality} along with these bounds, we obtain the following corollary, which shows that random variables in the conclusion of \cref{thm:nibbling-expectation} are indeed sufficiently well-concentrated around their expectation.

\begin{corollary} 
\label{corollary:nibbling-concentration}
Fix $0\leq i \leq t_\epsilon -1$. Suppose that the assumptions of \cref{thm:nibbling-expectation} are satisfied. For any $u,e,\gamma$, let $\mathcal{X}_{u}$ be the event that $\big{|}\hspace{1pt}|A_{i+1}(u)| - \Ex|A_{i+1}(u)|\hspace{1pt}\big{|} \geq \sqrt{\epsilon{d_i^{5/3}}}$, $\mathcal{Y}_{e}$ be the event that $\big|\hspace{1pt} |A_{i+1}(e)| - \Ex[|A_{i+1}(e)|]\hspace{1pt}\big| \geq \sqrt{{a_i^{5/3}}}$, and $\mathcal{Z}_{\gamma,u}$ be the event that $\big|\hspace{1pt}\deg_{i+1,\gamma}(u) - \Ex[\deg_{i+1,\gamma}(u)]\hspace{1pt}\big| \geq \sqrt{\epsilon{a_i^{5/3}}}$. Then,
\begin{align*}
\Pr[\mathcal{X}_{u}] &\leq 2\exp(-d_i^{1/3}) + \exp(-\epsilon d_i/10)  \\
\Pr[\mathcal{Y}_{e}] &\leq 2\exp(-a_i^{1/3}) + \exp(-\epsilon d_i/10)\\
\Pr[\mathcal{Z}_{\gamma,u}] &\leq 2\exp(-a_i^{1/3}) + \exp(-\epsilon d_i/10). 
\end{align*}
\end{corollary}
\begin{proof}
The variance bounds from \cref{thm:nibbling-variance-bounds} hold for any fixed $u,e,\gamma$ except with probability at most $\exp(-\epsilon d_i/10)$. Whenever these bounds hold, we apply \cref{grable:concentration-inequality} with $\varphi = d_i^{1/3}$ (in the case of $\mathcal{X}_{u}$) or $\varphi = a_i^{1/3}$ (in the case of $\mathcal{Y}_{e}$ and $\mathcal{Z}_{\gamma,u}$). Finally, we use that $2000a_{i}^{4/3} \leq a_{i}^{5/3}$ and $2000d_{i}^{4/3}\leq d_{i}^{5/3}$ since $\epsilon^{2}\Delta \geq D$ by assumption.  
\end{proof}

The next lemma is tailored for our application. Roughly speaking, we are given matchings $M_1,\ldots,M_s$ in the complete graph on $V(G)$. We wish to design a random procedure to properly color (most of) the edges of $G$ using $s$ colors in such a way that by considering the matchings $M'_1,\ldots,M'_s$ induced by each color class, the probability of any vertex $u$ becoming part of too many short cycles in any of the graphs $M_i\cup M'_i$ is sufficiently small.
\begin{lemma}
\label{prop:cycle-bound}
Let $M_1,\dots, M_s$ be a fixed collection of matchings in the complete graph on $V(G)$, where $s:=\Delta+1$. Let $0 < \beta < 1/10$. For any vertex $u$, let $C_{i}(u)$ denote the number of indices $\gamma$ for which $u$ lies on a cycle of length at most $\Delta^{\beta/2}$ in $M_\gamma\cup M'_\gamma$ by the end of round $i$ of the algorithm. Let $\mathcal{C}_{i,u}$ denote the event that $C_{i+1}(u) - C_{i}(u) \geq 1000\Delta^{\beta}/\epsilon^{4} $. 
Suppose that the assumptions of \cref{thm:nibbling-expectation} are satisfied. Then,
$$\Pr[\mathcal{C}_{i,u}]\leq \exp(-\Delta^{\beta}/\epsilon^{4}).$$
\end{lemma}
\begin{proof}
Fix $u \in V(G)$ as in the statement of the lemma. For each $\gamma \in [s]$ such that $u$ is not already lying on a cycle in $M_\gamma\cup M'_\gamma$ by the end of round $i$ of the algorithm, let $P_{\gamma}$ denote the unique maximal path (with a fixed, but otherwise arbitrary, orientation) in $M_\gamma\cup M'_\gamma$ containing $u$ at the end of round $i$. Observe that if the first and last edges of $P_\gamma$ do not belong to $M_\gamma$, then $P_\gamma$ cannot be extended to a cycle in $M_\gamma\cup M'_\gamma$. Let $p_\gamma$ denote the first vertex of $P_\gamma$ and let $q_\gamma$ denote its last vertex. For any vertex $v \in V(G)$, let $v^{\gamma}$ denote the unique (if it exists) vertex such that $\{v,v^\gamma\} \in M_\gamma$, and let $\overline{v}^{\gamma}$ denote the unique (if it exists) vertex such that the edge $\{v,\overline{v}^{\gamma}\}$ is colored $\gamma$ during the execution of the algorithm by the end of round $i+1$. Finally, for each $\gamma \in [s]$, consider the following sequence of vertices defined inductively: $w_{0,\gamma} := p_\gamma$, $w_{2i+1,\gamma} := \overline{w_{2i,\gamma}}^{\gamma}$ for $i\geq 0$, and $w_{2i,\gamma} := w_{2i-1,\gamma}^\gamma$ for $i\geq 1$. 

Note that $P_\gamma$ closes into a cycle of length at most $\ell=2\lceil\Delta^{\beta/2}/2\rceil + 1$ in $M_{\gamma}\cup M'_{\gamma}$ during the $(i+1)^{st}$ round only if one of the vertices $w_{1,\gamma},w_{3,\gamma}, w_{5,\gamma},\dots,w_{\ell,\gamma}$ is $q_\gamma$. In particular, 
at least one of the edges $\{w_{0,\gamma},q_\gamma\}, \{w_{2,\gamma},q_\gamma\},\dots,$ $\{w_{\ell-1,\gamma},q_\gamma\}$ must be \emph{tentatively} colored by $\gamma$ during the $(i+1)^{st}$ round. Letting $E_\gamma$ denote the random variable recording the number of such edges, it follows that $C_{i+1}(u) - C_i(u) \leq \sum_{\gamma\in[s]}E_\gamma$. Moreover, since a given edge $e$ is tentatively colored by a given color $\gamma$ during the $(i+1)^{st}$ round with probability at most $1/|A_i(e)| \leq 2/a_i \leq 2\Delta/d_{t_\epsilon}^2 \leq 512/\Delta \epsilon^{4}$, and since the tentative colors for different edges are chosen independently, we see that conditioning on any choice for the collection of vertices $\{w_{2,\gamma}\notin P_\gamma,\dots,w_{\ell-1,\gamma}\notin P_\gamma\}_{\gamma \in [s]}$, the random variable $\sum_{\gamma \in [s]}E_\gamma$ is stochastically dominated by the random variable $\Bin(\Delta^{1+(\beta/2)},512/\Delta\epsilon^{4})$. Therefore, by Chernoff's bound for the binomial distribution followed by the law of total probability to remove the conditioning, it follows that 
$$\Pr[\mathcal{C}_{i,u}] \leq \Pr[\Bin(\Delta^{1+(\beta/2)},512/\Delta\epsilon^{4})\geq 1000\Delta^{\beta}/\epsilon^{4}]\leq \exp(-\Delta^{\beta}/\epsilon^{4}),$$
which completes the proof. 
\end{proof}

The following lemma combines \cref{corollary:nibbling-concentration} and \cref{prop:cycle-bound} to prove the existence of a `good' outcome of a given round of the algorithm.

\begin{lemma} 
\label{lemma:one-step}
Fix $0\leq i \leq t_{\epsilon}-1$. Suppose that the assumptions of \cref{prop:cycle-bound} are satisfied and $\Delta^{-1/B} \leq \epsilon < 1/100$ for some $B > 20$. 
Then with positive probability, the following holds at the end of stage $i+1$ for all vertices $u$, edges $e$ and colors $\gamma$ simultaneously:
\begin{align*}
|A_{i+1}(u)| & =(1\pm e_{i+1})d_{i+1}\\
|A_{i+1}(e)| & =(1\pm e_{i+1})a_{i+1}\\
\deg_{i+1,\gamma}(u) & =(1\pm e_{i+1})a_{i+1}\\
C_{i+1}(u) - C_i(u) &\leq 1000\Delta^{\beta}/\epsilon^{4},
\end{align*}
where  $0 < \beta < \frac{1}{10} - \frac{1}{B}$ is fixed, and
\begin{equation}
\label{eqn:recursion-error}
e_{i+1}=\kappa\left(e_{i}+\sqrt{a_{i}^{-1/3}}\right)\leq \kappa\left(e_{i}+(1-p_\epsilon)^{-i}\Delta^{-1/6}\right), 
\end{equation}
with $\kappa = 1+K\epsilon$.  
\end{lemma}
\begin{proof}
Let $\mathcal{X}_{u}$, $\mathcal{Y}_{e}$, $\mathcal{Z}_{\gamma,u}$ be the events defined in \cref{corollary:nibbling-concentration}, and let $\mathcal{C}_{u}:=\mathcal{C}_{i,u}$ be the event defined in \cref{prop:cycle-bound}. It suffices to show that 
$$\Pr\left[\left(\bigcap_{u\in V(G)}\mathcal{X}_{u}^{c}\right)\cap\left(\bigcap_{e\in E(G)}\mathcal{Y}_{e}^{c}\right)\cap \left(\bigcap_{\gamma\in[s],u\in V(G)}\mathcal{Z}_{\gamma,u}^{c}\right)\cap \left(\bigcap_{u\in V(G)}\mathcal{C}_{u}^{c}\right)\right]>0.$$ 
We will show this using the symmetric local lemma. To this end, we note that two events of the form $\mathcal{E}(u,e,\gamma)$ and $\mathcal{E}(u',e',\gamma')$ can depend on each other only if at least one of $u$ or $e$ is within distance at most (say) $4\Delta^{\beta/2}$ from one of $u'$ or $e'$. Since the maximum degree of $G$ is $\Delta$, it follows that the dependency graph of the events listed above has maximum degree at most (say) $\text{poly}(\Delta)\Delta^{4\Delta^{\beta/2}}\leq \exp(\Delta^{3\beta/4})$, where the last inequality holds for all $\Delta$ sufficiently large. Also, by \cref{prop:cycle-bound}, events of the form $\mathcal{C}_{u}$ hold with probability at most $\exp(-\Delta^{\beta})$, whereas by \cref{corollary:nibbling-concentration}, the other events hold with probability at most $\exp(-\epsilon d_{t_\epsilon}/10) + 2\exp\left(-a_{t_\epsilon}^{1/3}\right)$. Hence, if this latter quantity were much less than $\exp(\Delta^{3\beta/4})$, we would be done. This is indeed true provided that $\beta < \frac{1}{10} - \frac{1}{B}$ and $\Delta$ is sufficiently large. 
\end{proof}

Finally, we iterate \cref{lemma:one-step} to prove the main result of this appendix.

\begin{proposition} 
\label{prop:good-nibbles}
There exist constants $B,D > 20$ for which the following holds. Let $G$ be a graph with maximum degree $\Delta$ and minimum degree $\delta$ such that $\Delta - \delta \leq \Delta^{5/6}$ and $\Delta \geq D$. Let $M_1,\dots,M_s$ be a fixed collection of matchings in the complete graph on $V(G)$, where $s:=\Delta + 1$. Then, for any fixed $ \Delta^{-1/B}\leq \epsilon < 10^{-4}$, the following holds with positive probability for the execution of \cref{algorithm:nibble} on $G$ with parameter $\epsilon$ for $t_{\epsilon}$ stages: for fixed $0 < \beta < \frac{1}{10} - \frac{1}{B}$, for all $0 \leq i \leq t_\epsilon-1$, and for all vertices $u$, all edges $e$, and all colors $\gamma$,
\begin{itemize}
\item $|A_i(u)| = (1\pm \epsilon^{3}) d_i$
\item $|A_i(e)| = (1 \pm \epsilon^{3}) a_i$
\item $\deg_{i,\gamma}(u) = (1\pm \epsilon^{3}) a_i$
\item $C_{i+1}(u) - C_i(u) \leq 1000\Delta^{\beta}/\epsilon^{4}$. 
\end{itemize}
In particular, for every vertex $u$, $C_{t_\epsilon}(u) \leq 1000t_\epsilon\Delta^{\beta}/\epsilon^{4} \leq \Delta^{\beta}/\epsilon^{6}$. Further, the number of uncolored edges incident to $u$ is at most $|A_{t_\epsilon}(u)| \leq 2d_{t_\epsilon} \leq \epsilon\Delta/2$.
\end{proposition}

\begin{remark} By taking the matchings $M'_1,\dots, M'_{\Delta}$ to be the edges colored (at the end of stage $t_\epsilon$) by $1,\dots,\Delta$ respectively, it is immediately seen that the above proposition, with $\epsilon = \Delta^{-1/B}$ and $\beta = 1/20$, implies \cref{lemma:avoiding short cycles}. 
\end{remark}

\begin{proof}
We view the execution of \cref{algorithm:nibble} as a branching process, where in each round of the algorithm, we branch out according to which edges are assigned final colors, and which final colors are assigned to these edges. Generate this tree for $t_\epsilon$ levels, and consider any root to leaf path such that for each intermediate `branch', the endpoint further from the root satisfies the conclusions of \cref{lemma:one-step} given the parameters at its parent. Such a root-to-leaf path is guaranteed to exist by \cref{lemma:one-step}. To complete the proof, we track the error introduced by the iterative application of \cref{lemma:one-step}, and show that it is no more than what is stated in the proposition. 

Setting $A:= \Delta^{-1/6}$ and $P:=(1-p_\epsilon)^{-1}$, we get from \cref{eqn:recursion-error} that 
$$e_\ell \leq \kappa^{\ell}e_0 + A[\kappa^\ell + \kappa^{\ell-1}P+\dots + \kappa P^{\ell-1}]$$
for all $0\leq \ell \leq t_{\epsilon}$. Since $P = (1-p_\epsilon)^{-1}$ is also of the form $1+ K'\epsilon$ for some constant $K' > 0$, it follows that 
$$e_\ell \leq (1+L\epsilon)^\ell e_0+ \ell(1+L\epsilon)^\ell\Delta^{-1/6},$$
where $L = \max\{K,K'\}$. By \cref{rmk:init-error}, $e_0 \leq (\Delta-\delta)/\Delta \leq \Delta^{-1/6}$. Therefore, 
$$ e_\ell \leq 2\ell \exp(L\epsilon \ell) \Delta^{-1/6}.$$
The right hand side is maximized when $\ell = t_\epsilon$, in which case it is at most 
$$\left(\frac{1}{\epsilon}\right)^{3L}\Delta^{-1/6} \leq \epsilon^{3},$$
where the last inequality holds provided we take $B \geq 18(L+1)$.
\end{proof}
\end{appendix}

\end{document}